\documentclass[leqno,11pt]{amsart}
\usepackage[colorlinks,pagebackref,hypertexnames=false]{hyperref}

\usepackage{amsmath,amsthm,amssymb,mathrsfs,systeme} 
\usepackage[alphabetic,backrefs]{amsrefs}
\usepackage{ae,aecompl}
\usepackage[margin=1in]{geometry} 
\usepackage[retainorgcmds]{IEEEtrantools}

\usepackage[english]{babel}

\usepackage{bookmark}

\usepackage{amsmath,amsfonts,amssymb,graphics,graphicx,latexsym,amscd,subfigure,hyperref}
\usepackage[usenames]{xcolor}
\usepackage[retainorgcmds]{IEEEtrantools}
\usepackage{graphicx}
\usepackage{amssymb}
\usepackage{xypic}
\usepackage{hyperref}
\usepackage{amsmath,amssymb}
\usepackage{dsfont}

\numberwithin{equation}{section}

\newcommand{\bbR}{\mathbb{R}}

\newcommand{\bbC}{\mathbb{C}}

\newcommand{\bbH}{\mathbb{H}}

\newcommand{\rmq}{\rmq}

\renewcommand{\epsilon}{\varepsilon}

\newcommand{\Bl}{\mathrm{Bl}}

\def\dist{\mathrm{dist}}

\newtheorem{theorem}{Theorem}

\newtheorem{lemma}{Lemma}

\theoremstyle{definition} \newtheorem{definition}{Definition}

\newtheorem{remark}{Remark}

\author[Gon\c calo Oliveira]{Gon\c calo Oliveira}
\address{Gon\c calo Oliveira, Department of Mathematics and Center for Mathematical Analysis, Geometry and Dynamical Systems, Instituto Superior T\'ecnico, Lisbon,\\ Av. Rovisco Pais, 1049-001 Lisboa,  Portugal}
\email{goncalo.m.f.oliveira@tecnico.ulisboa.pt}
\email{galato97@gmail.com}

\author[Rosa Sena-Dias]{Rosa Sena-Dias}
\address{Rosa Sena-Dias, Department of Mathematics and Center for Mathematical Analysis, Geometry and Dynamical Systems, Instituto Superior T\'ecnico, Lisbon,\\ Av. Rovisco Pais, 1049-001 Lisboa,  Portugal}
\email{rsenadias@math.ist.utl.pt}

\begin{thanks}
{Partially supported by the Funda\c{c}\~{a}o para a Ci\^{e}ncia e a Tecnologia (FCT/Portugal) through project EXPL/MAT-PUR/1408/2021EKsta.}
\end{thanks}

\title[]{Scalar-flat K\"ahler metrics with varying cone angle singularities along a divisor} 

\begin{document}

\begin{abstract}
	Using an ansatz due to LeBrun we construct complete scalar-flat K\"ahler metrics with a prescribed varying conical singularity along a divisor.
\end{abstract}

\maketitle

\tableofcontents

\section{Introduction and main result}

Cone angle K\"ahler metrics have recently come into the spotlight following the role they have played in the proof of the Yau-Tian-Donaldson conjecture. By using a continuity method for a family of singular metrics with a cone angle along a divisor, Chen-Donaldson-Sun have proved the existence of smooth K\"ahler-Einstein metrics on compact K\"ahler manifolds which are K-stable (see \cite{cds}). In \cite{dconeangle}, Donaldson developed a framework for working with cone angle K\"ahler metrics as they arise in the program carried out in \cite{cds} but such metrics had been considered before (see \cite{j} or \cite{m} for instance). There are also explicit constructions of K\"ahler metrics with cone angles. In \cite{b} Bourbon constructs Ricci-flat K\"ahler metrics with cone angles using the Gibbons-Hawking ansatz whereas in \cite{bs} Borbon-Spotti construct K\"ahler-Einstein metrics with a cone angle via the Calabi ansatz.

In \cite{LeBrun}, LeBrun proves the existence of self-dual metrics on blow-ups of $\bbC^2$ by using an ansatz for scalar-flat K\"ahler metrics with an $S^1$-symmetry generalising the Gibbons--Hawking's ansatz \cite{GH}. Although LeBrun's ansatz is given as a means to construct specific examples, it holds in great generality and has not been fully exploited. 

In this note, we use LeBrun's ansatz to construct scalar-flat K\"ahler metrics on arbitrary blow-ups of $\bbC^2$ with a varying cone angle along a divisor. Our goal is to take a small step towards answering a question of Atiyah-LeBrun (see \cite{Atiyah}) concerning the existence of Einstein metrics with varying cone angles. 

Our main result, Theorem \ref{thm:Main}, states that it is possible to find zero scalar curvature K\"ahler metrics with a prescribed varying cone angle along a divisor $D$. We shall say that any such metric with cone singularities along $D$ is complete if any geodesic not intersecting $D$ can be continued for all time. 

\begin{theorem}\label{thm:Main}
	Let $k \in \mathbb{N}_0$, $\xi_1, \ldots , \xi_k \in \mathbb{C} \times \lbrace 0 \rbrace \subset \mathbb{C}^2$ and $\Bl_{\xi_1, \ldots , \xi_k}\mathbb{C}^2$ be the blow-up of $\mathbb{C}^2$ at the points $\xi_1, \ldots , \xi_k$. 
	
	Then, for any smooth positive function {over $\mathbb{CP}^1 \times \lbrace 0 \rbrace=\left( \mathbb{C} \cup \lbrace \infty \rbrace\right) \times \lbrace 0 \rbrace$}, there is a complete scalar-flat K\"ahler metric on $\Bl_{\xi_1, \ldots , \xi_k}\mathbb{C}^2$ which has cone angle $\beta$ along the strict transform of $\mathbb{C} \times \lbrace 0 \rbrace$.
\end{theorem}

In particular, if $k=0$, Theorem \ref{thm:Main} constructs complete scalar-flat K\"ahler metrics on $\mathbb{C}^2$ with prescribed cone angle along $\mathbb{C} \times \lbrace 0 \rbrace$.

{In fact, given that these metrics are K\"ahler and scalar-flat, they are also anti-self-dual. To our knowledge these are also the first example of anti-self-dual metrics with varying cone angle along a $2$-dimensional submanifold. Furthermore, just as in \cite{LeBrun}, we can conformally change these metrics so that they compactify, thus yielding anti-self-dual metrics with a varying cone angle on compact manifolds. }

{\noindent \textbf{Acknowledgments:} We would like to thank Joel Fine for his interest in this work and for useful suggestions.}

\section{LeBrun's hyperbolic ansatz}

Recall that LeBrun's ansatz for scalar-flat K\"ahler metrics with circle symmetry can be written in terms of two functions $v$ and $W$. It yields a K\"ahler metric of the form
\begin{align*}
	g & = W^{-1} \eta^2 + W dx_1^2 + W e^v (dx_2^2 + dx_3^2 ) \\
	\omega & = dx_1 \wedge \eta + W e^v dx_2 \wedge dx_3,
\end{align*}
with $\eta$ a connection form on a circle bundle $\hat{X}^4$ over an open set $U \subset \mathbb{R}^3$ defined in terms of $v$ and $W$. This connection is such that its curvature is
$$d \eta = \partial_{x_1} (We^v) dx_2 \wedge dx_3 + \partial_{x_2}W dx_3 \wedge dx_1 + \partial_{x_3}W dx_1 \wedge dx_2 . $$
In particular, from $d^2 \eta =0$ we have
\begin{equation}\label{eq:W compatibility}
	\partial_{x_1}^2 (We^v) + \partial_{x_2}^2 W + \partial_{x_3}^2 W = 0 .
\end{equation}
Such a K\"ahler metric has scalar curvature given by
$$s= - \frac{\partial_{x_1}^2 (e^v) + \partial_{x_2}^2 v + \partial_{x_3}^2 v}{ W e^v }.$$
LeBrun's hyperbolic ansatz consists in setting $v(x_1,x_2,x_3)=\log (2x_1)$ in which case $\partial^2_{x_1}(e^v)=0$ and also $\partial_{x_2}v=0=\partial_{x_3}v$ so the metric is scalar-flat. Furthermore $\partial_{x_1}^2(e^vW)= \partial_{x_1}^2 (2x_1 W ) = \partial_{x_1}^2 (2x_1 W )$ and so the compatibility equation \ref{eq:W compatibility} becomes
\begin{align*}
	0 & = \partial_{x_1}^2 (We^v) + \partial_{x_2}^2 W + \partial_{x_3}^2 W \\
	& = \partial_{x_1}^2 (2x_1W) + \frac{\partial_{x_2}^2 (2x_1W) + \partial_{x_3}^2 (2x_1W)}{2x_1}  \\
	& = \Delta_{h} V,
\end{align*}
where $V=2x_1W$ and $h=\frac{dx_1^2}{(2x_1)^2} + \frac{1}{2x_1} (dx_2^2+dx_3^2)$. Therefore, if such a $V$ solves $\Delta_{h}V=0$, the equation for $\eta$ is (up to gauge) 
\begin{align*}
	d \eta & = \partial_{x_1} (We^v) dx_2 \wedge dx_3 + \partial_{x_2}W dx_3 \wedge dx_1 + \partial_{x_3}W dx_1 \wedge dx_2 \\
	& = \partial_{x_1} (2x_1W) dx_2 \wedge dx_3 + \partial_{x_2}(2x_1 W) dx_3 \wedge \frac{dx_1}{2x_1} + \partial_{x_3}(2x_1W) \frac{dx_1}{2x_1} \wedge dx_2 \\
	& = \ast_{h} dV,
\end{align*}
which can be solved for $\eta,$ a connection on a line bundle over $U,$ provided that $\ast_{h} dV$ has integer periods.

\begin{remark}
	Writing $z=\sqrt{2x_1}$, the metric $h$ becomes
	\begin{equation}\label{eq:hyperbolic metric}
		h=\frac{dz^2+dx_2^2+dx_3^2}{z^2},
	\end{equation}
	and so is the hyperbolic metric on
	$$\bbH=\lbrace (z,x_2,x_3) \in \mathbb{R}^3 \ | \ z>0 \rbrace. $$
\end{remark}

We can restate LeBrun's result as follows. 

\begin{theorem}[LeBrun in \cite{LeBrun}]\label{thm:LeBrun}
	Let $h$ denote the hyperbolic metric on the upper half plane. Let $U \subset \bbH$ and $V$ a positive $h$-harmonic function on $U$ such that $\ast_{h_v} dV$ represents an integral class in $\bbH^2_{dR}(U, \mathbb{R})$. Then, there is a circle bundle $M \to U$ and a connection $1$-form $\eta$ over $U$ satisfying $d\eta = \ast_{h_v} dV$. Furthermore, the metric and symplectic form
	\begin{align*}
		g & = z^2 \left( V^{-1} \eta^2 + V h \right), \\
		\omega & = z dz \wedge \eta + V dx_2 \wedge dx_3,
	\end{align*}
	equip $M$ with a scalar-flat K\"ahler metric.
\end{theorem}

In \cite{LeBrun} the author goes on to study the metrics obtained by setting
$$V=1+\sum_{i=1}^k G_{p_i},$$
where $p_1, \ldots , p_l \in \bbH$ are all distinct and each $G_{p_i}$ denotes the Green's functions based at $p_i$ which converges to zero at infinity, i.e.
$$\Delta_{h}V=\sum_{i=1}^k  \delta_{p_i} , \ \ V \to 1\, \text{at} \,\infty .$$
In this situation $U=\bbH \backslash \lbrace p_1, \ldots , p_k \rbrace$ and $M$ is a circle bundle over $U$. The metric $g$ is not complete on $M$ but it extends smoothly to a complete metric on another manifold $\overline{M} \supset M$ obtained by adding points above each $p_i$ and an $\mathbb{R}^2$ along $z=0$.
If all the points $p_i=(a_i,b_i,c_i)$ lie in different vertical $z$-lines, i.e. if the $(a_i,b_i)$ are all different, then $\overline{M}$ is simply $Bl_{\xi_1, \ldots , \xi_k}\mathbb{C}^2$ where $\xi_i=(a_i+ib_i,0)$ for $i=1, \ldots , k$. On the other hand, if two or more of the points $p_i$ lie in the same vertical line, then $\overline{M}$ can be obtained by an iterated blow-up. For a proof of these facts see pages 233 and 234 in \cite{LeBrun}.

\section{Proof of the main result}

\subsection{The Dirichlet problem on hyperbolic space}

Recall that the sphere $S(\infty)$ at infinity in hyperbolic space $\bbH$ can be defined as the set of equivalence classes of geodesic rays emanating from some fixed interior point. The equivalence relation is given by declaring that two rays $\gamma_1,\gamma_2 : \mathbb{R}^+_0 \rightarrow \bbH$ are equivalent if there is $C>0$ such that $\dist(\gamma_1(t),\gamma_2(t)) < C <+\infty$ for all $t>0$. 

Then, one can compactify $\bbH$ by adding the sphere at infinity $\overline{\bbH}=\bbH \cup S(\infty)$ and solve the Dirichlet problem on $\overline{\bbH}$ for the hyperbolic Laplacian (see \cite{Andersson}, \cite{Sullivan} and also \cite{AnderssonSchoen}).

\begin{theorem}[\cite{Andersson}, \cite{Sullivan}]\label{thm:Dirichlet}
	Let $\beta:S(\infty) \to \mathbb{R^{+}}$ be a continuous function. Then, there is a unique $u_\beta: {\bbH} \to \mathbb{R}$ solving
	\begin{align*}
		& \Delta_h u_\beta  =0 , \\
		& u_\beta |_{S(\infty)}  =\beta^{-1} .
	\end{align*}
\end{theorem}

We shall now prove a result which will prove useful in analysing the local behaviour of the metrics which we shall construct.

\begin{lemma}\label{lemmaubeta}
	Let $\beta \in C^{\infty}(S(\infty))$ be nowhere vanishing and $u_\beta$ be given by Theorem \ref{thm:Dirichlet}. Then, there is $\epsilon >0$ and a smooth bounded function $v$ on $T_\epsilon = \lbrace q \in \bbH | \ z(q) \leq \epsilon \rbrace$ such that
	$$u_\beta = \beta^{-1} + zv \text{ in $T_\epsilon$,} $$
	where $z$ denotes the first coordinate in $\bbH$ as before and we extend $\beta$ to $T_\epsilon$ by making it constant in $z$.
\end{lemma}
\begin{proof}
	First we shall prove that there is $C>0$ such that the functions $\varphi_\pm = \beta^{-1} \pm Cz$ are respectively super(sub)-harmonic. This follows from the computation, using the upper-half-space model in which $z^2h=dx_2^2+dx_3^2+dz^2.$ We have
	\begin{align*}
		\Delta_h \varphi & = -z^2 (\partial_{x_2}^2 \varphi + \partial_{x_3}^2 \varphi + \partial_z^2 \varphi ) +z \partial_z \varphi .
	\end{align*}
	Applying this to $\varphi_\pm$ gives
	\begin{align*}
		\Delta_h \varphi_\pm & = z \left( \pm C -z\left( \partial_{x_2}^2 \beta^{-1} + \partial_{x_3}^2 \beta^{-1}\right) \right) ,
	\end{align*}
	which is positive/negative for $|C| \gg 0$ depending on whether we choose the positive/negative sign respectively. This finishes the proof that $\varphi_+$ is superharmonic and $\varphi_-$ is subharmonic.
	
	Next we show that it is possible to adjust $C$ so that
	\begin{equation}\label{eq:Comparison}
		\varphi_+ |_{\partial \overline{T_\epsilon}} \geq u_\beta |_{\partial \overline{T_\epsilon}} \geq \varphi_- |_{\partial \overline{T_\epsilon}}, 
	\end{equation}
	where $\overline{T_\epsilon}$ denotes the closure of $T_\epsilon$ in $\overline{\bbH}$. We have $\partial \overline{T_\epsilon} = S(\infty) \cup \lbrace z= \epsilon \rbrace$ and we shall analyse each of these components separately. In $S(\infty),$  $\varphi_\pm = \beta^{-1} = u_\beta$ and so \ref{eq:Comparison} holds. As for when $z=\epsilon$, we have $\varphi_{\pm} = \beta^{-1} \pm C \epsilon$ and by possibly increasing $C$ and decreasing $\epsilon$
	$$\beta^{-1} + C \epsilon \geq u_{\beta} |_{z=\epsilon} \geq \beta^{-1} - C \epsilon.$$
	Given that $u_\beta$ is harmonic and takes the values of $\beta^{-1}$ at $S(\infty)$ we find from the maximum principle that $\sup_{{\bbH}} u_\beta = \max_{S(\infty)} \beta^{-1}$ and $\inf_{\bbH} u_\beta = \min_{S(\infty)} \beta^{-1}$. Thus, it is enough to ensure that $C$ is big enough to satisfy
	$$\min_{S(\infty)} \beta^{-1} + C \epsilon \geq \max_{S(\infty)} \beta^{-1} \geq \min_{S(\infty)} \beta \geq \max_{S(\infty)} \beta^{-1} - C \epsilon,$$
	which can be achieved by setting $C> \frac{\max_{S(\infty)} \beta^{-1} - \min_{S(\infty)} \beta^{-1}}{\epsilon}$.
	
It follows from the maximum principle that inside $T_\epsilon$ we have
	$$\beta^{-1} - Cz \leq u_\beta \leq \beta^{-1} + C z, $$
	and the result follows.
\end{proof}

\begin{remark}
In fact, it seems to be known that $u_\beta$ can be chosen to be smooth on $\overline{\bbH}$ and this would be more than enough for our purposes (see \cite{t}). But such a result is much harder than our lemma so we thought it worthwhile to give a proof of what we need.
\end{remark}

\subsection{Local behaviour of the metric along the strict transform of $\mathbb{C} \times \lbrace 0 \rbrace$}

The K\"ahler structures which we shall study are those constructed by Theorem \ref{thm:LeBrun} setting 
\begin{equation}\label{eq:V}
	V=u_\beta + \sum_{i=1}^l G_{p_l},
\end{equation}
where $u_\beta$ is the solution of the Dirichlet problem in Theorem \ref{thm:Dirichlet} for a positive function $\beta$ and $G_{p_i}$ are the Green's functions associated with the points $\lbrace p_i \rbrace_{i=1, \ldots , l} \subset \bbH$ as in the discussion following Theorem \ref{thm:LeBrun}. Note that $u_\beta$ is positive since $\inf_{{\bbH}} u_\beta = \min_{S(\infty)} \beta^{-1}$ therefore $V$ is positive.

The same analysis as that in pages 231 and 232 of \cite{LeBrun} shows that such structure extends smoothly over points $\lbrace \tilde{p_i} \rbrace_{i=1, \ldots ,l}$ added to $M$ above the points $\lbrace p_i \rbrace_{i=1, \ldots , l} $. Hence, we shall now analyse the local behaviour of these metrics near $z=0$.

Given that $G_{p_i}$ decay to zero with $z^2$, the function $V$ has the same asymptotic behaviour as $u_\beta$ and near $z=0$. In order to show that $g$ has cone angle $\beta$ along $z=0$ it is enough to show that the conformally related $V^{-1}g$ has cone angle $\beta$ along $z=0$ (recall that $V$ is smooth near $S(\infty)$ and uniformly bounded below by $\min_{S(\infty)} \beta$). Now
\begin{align*}
	V^{-1}g & = z^2 V^{-2} \eta^2 + z^2h  \\
	& = z^2 V^{-2} \eta^2 + dz^2 + g_{\mathbb{R}^2} \\
	& = z^2 \left(\beta^{-1} + z w\right)^{-2} \eta^2 + dz^2 + g_{\mathbb{R}^2} \\
	& = \beta^2 z^2 \eta^2 + dz^2 + g_{\mathbb{R}^2} + z^{3} \mu,
\end{align*}
for some smooth symmetric $2$-tensor $\mu$ which is bounded with respect to $\beta^2 z^2 \eta^2 + dz^2 + g_{\mathbb{R}^2}$ and some smooth function $w$. This is precisely the form a metric with cone singularities $\beta$ along $z=0$, see page 101 in \cite{Jeffers} and also page 2 in \cite{Atiyah}.
For completeness we note that we have a holomorphic coframing
$$\alpha_1 = dx+idy , \text{ and } \alpha_2= dz+i zV^{-1} \eta,$$
and the K\"ahler form satisfies
$$V^{-1}\omega = \frac{i}{2} \alpha_1 \wedge \overline{\alpha_1} +  \frac{i}{2} \alpha_2 \wedge \overline{\alpha_2}. $$

\subsection{Asymptotic behaviour of the metrics and their completeness}\label{ss:asymptotic}

We shall start with the definition

\begin{definition}
	Let $\beta : \mathbb{C} \to \mathbb{R}^+$. We define the standard warped product metric on $\mathbb{C}^2$ with cone angle $\beta$ along $\mathbb{C} \times \lbrace 0 \rbrace$ to be the metric given by
	$$g_\beta= dz^2 + \beta^2 z^2 d \theta^2 + g_{\mathbb{C}},$$
	with $(z, \theta)$ polar coordinates on the first factor of $\mathbb{C}^2 = \mathbb{C} \times \mathbb{C}$ and $g_{\mathbb{C}}$ the Euclidean metric on the second.
\end{definition}

\begin{remark}\label{rem:Kahler}
	Notice that, as we define it, the standard warped product metric above is not K\"ahler. As it will become apparent below, this is to be expected as our metrics will only be asymptotically conformal to such metrics. Indeed, one can immediately check that the conformally transformed metric $\beta^{-1} g_\beta$ is indeed K\"ahler as its associated $2$-form is given by $dz \wedge z d \theta + \beta dx_2 \wedge dx_3$ where $(x_2,x_3)$ are coordinates on the second factor of $\mathbb{C}^2$.
\end{remark}

Consider first the case when $k=0$, i.e. $V=u_\beta$. In this situation we obtain a scalar-flat K\"ahler metric on $\mathbb{C}^2$ with a varying cone singularity $\beta$ along $\mathbb{C} \times \lbrace 0 \rbrace$. By the maximum principle $u_\beta$ is uniformly bounded above and below away from zero therefore the metric $g$ is quasi-isometric to $u_\beta^{-1}g$ which can be written as
$$u_\beta^{-1}g = \frac{z^2}{u_\beta^2} \eta^2 + dz^2 + g_{\mathbb{R}^2}. $$
We can write $\eta=d\theta + A$ for $\theta$ a coordinate along the circle fibres and $A.$
{
\begin{lemma} 
In the setting above, {if $\beta^{-1} \in C^{1,\delta}(S(\infty)),$} the $1$-form $A$ can be chosen so as to decay as $(z,x_2,x_3)$ goes to infinity. 
\end{lemma}
\begin{proof}
Recall that $d A = \ast_h dV$ where $V=u_\beta+\sum_{i=1}^k G_{p_i}$.  Let us choose $A$ to be in Coulomb Gauge i.e.
$$
\begin{cases}
d \ast_h A_\beta =0,\\
d A = \ast_h dV.
\end{cases}
$$
By linearity $A=A_\beta+A_0$ where $d A_\beta = \ast_h du_\beta$ and $d A_0 = \ast_h d\sum_{i=1}^k G_{p_i}$.
In \cite{LeBrun}, Lebrun uses asymptotic properties for the Green's function on hyperbolic space to show that the lemma holds for $A_0,$ when in Coulomb gauge. Hence, it is enough to prove the result for $A_\beta$ in Coulomb gauge, i.e. $d \ast_h A_\beta =0$.

We start by observing that, by Proposition 6.3 in \cite{t}, if $\beta^{-1} \in C^{1,\delta}(S(\infty))$, the euclidean norm of $du_\beta$ is bounded. More precisely consider the disk model for the hyperbolic metric:
$$
h=\frac{dy_1^2+dy_2^2+dy_3^2}{1-\left(y_1^2+y_2^2+y_3^2\right)},
$$
where $(y_1,y_2,y_3)$ take values on the ball of radius $1$ in $\bbR^3.$ There is $C$ such that
$$| d u_\beta |_E \leq C ,$$
where $|\cdot|_E$ is the Euclidean metric $dy_1^2+dy_2^2+dy_3^2$ on the ball. This implies that $|\partial_{y_i} u_\beta|<C$ for $i=1,2,3$. The relation between the upper half plane model and the disk model for the hyperbolic metric gives
\begin{eqnarray}
	y_1 & = \frac{x_2^2+x_3^2-1}{(z+1)^2+x_2^2+x_3^2} \\
	y_2 & = \frac{2x_2}{(z+1)^2+x_2^2+x_3^2} \\
	y_3 & = \frac{2x_3}{(z+1)^2+x_2^2+x_3^2} .
\end{eqnarray}
A direct computation shows  that there is $C'$ such that $$|\partial_{x_j} y_i|+|\partial_z y_i| \leq\frac{C'}{ (z+1)^2+x_2^2+x_3^2},$$ for all $i \in \lbrace 1,2,3 \rbrace$ and $j\in \lbrace 1,2 \rbrace $. As a consequence, we have that there is $C''$ such that
$$|\partial_z u_\beta| + |\partial_{x_2} u_\beta| + |\partial_{x_3} u_\beta| \leq  \frac{C''}{(z+1)^2+x_2^2+x_3^2} ,$$
and thus
$$|dA_\beta|= z^2 |dA_\beta|_h=z^2 |du_\beta|_h = z |du_\beta| \leq \frac{C^{(3)}z}{(z+1)^2+x_2^2+x_3^2} ,$$
where $C^{(3)}$ is some constant and $|\cdot|$ denotes the norm with respect to the Euclidean metric $dz^2+dx_2^2+dx_3^2$ in the upper half-space. This together with standard elliptic estimates on the operator $d + d \ast$ then shows that, in Coulomb gauge, $A_\beta$ satisfies 
$$|A_\beta| \leq \frac{C^{(4)}z}{(z+1)^2+x_2^2+x_3^2} ,$$
where $C^{(4)}$ is some constant and thus decays as either $z$ or $x_2^2+x_3^2$ go to infinity.

\end{proof}

We use the above lemma and the fact that $u_\beta = \beta^{-1} + zv$ for some bounded $v$ (Lemma \ref{lemmaubeta}) to see that} for large $z>Z_0 \gg 1$ or $x_2^2+x_3^2 \geq R_0^2 \gg 1$ we can write the quasi-isometric metric $u_\beta^{-1}g$ as
$$dz^2 + \beta^2 z^2 d \theta^2 + g_{\mathbb{R}^2},$$
up to lower order terms in $z^2$ or $x_2^2+x_3^2$. This is precisely what we called the standard warped product metric on $\mathbb{C}^2$ with varying cone angle $\beta$ along $\mathbb{C} \times \lbrace 0 \rbrace$. Hence, for large $z^2$ or $x_1^2+x_3^2$ the metric $g$ is quasi-isometric to $g_\beta$ and so the length of geodesics for $g$ can be uniformly controlled by the length of geodesics for $g_\beta$ and vice-versa. This shows that $g$ is complete (in the sense defined immediately before the statement of Theorem \ref{thm:Main}) if and only if $g_\beta$ itself is. On the other hand, the fact that $g_\beta$ is complete can be deduced from the fact that it is itself quasi-isometric to the Euclidean metric.

We have thus proved the following.

\begin{lemma}\label{lem:Asymptotic 0}
	Let $\beta: \mathbb{C} \times \lbrace 0 \rbrace \to \mathbb{R}^+$ be a smooth function which extends as a positive function over $\mathbb{CP}^1 \times \lbrace 0 \rbrace=\left( \mathbb{C} \cup \lbrace \infty \rbrace\right) \times \lbrace 0 \rbrace.$ Assume that {$\beta^{-1} \in C^{1,\delta}(S(\infty))$} and { let $(g, \omega)$} be the scalar-flat K\"ahler metric on $\mathbb{C}^2$ from Theorem \ref{thm:LeBrun} with $V=u_\beta$. Then $g$ is asymptotically quasi-isometric to the standard warped product metric on $\mathbb{C}^2$. In particular, it is complete.
\end{lemma} 

As for the case when $k \geq 1$ we have $V= u_\beta + \sum_{i=1}^k G_{p_i}$ and given that $G_{p_i}$ decays exponentially with respect to the hyperbolic distance to some interior point in $\bbH$ we find that $V \sim u_\beta + \ldots$ where $\ldots$ denotes lower order terms. This shows the following.

\begin{lemma}
	Let $\beta: \mathbb{C} \times \lbrace 0 \rbrace \to \mathbb{R}^+$ be a smooth function which extends continuously as a positive function  over $\mathbb{CP}^1 \times \lbrace 0 \rbrace=\left( \mathbb{C} \cup \lbrace \infty \rbrace\right) \times \lbrace 0 \rbrace$ and $\xi_1, \ldots , \xi_k \in \mathbb{C} \times \lbrace 0 \rbrace \subset \mathbb{C}^2.$  {Assume $\beta^{-1} \in C^{1,\delta}(S(\infty))$.} Let $(g, \omega)$ be the scalar-flat K\"ahler metric on $Bl_{\xi_1, \ldots ,\xi_k}\mathbb{C}^2$ from Theorem \ref{thm:LeBrun} with $V=u_\beta+ \sum_{i=1}^k G_{p_i}$. Then $g$ is asymptotic to the metric $g_0$ from Lemma \ref{lem:Asymptotic 0}. In particular, it is complete.
\end{lemma}

{Our main Theorem \ref{thm:Main} follows from this lemma since if $\beta: \mathbb{C} \times \lbrace 0 \rbrace \to \mathbb{R}^+$ extends smoothly as a positive function to $S(\infty),$ then $\beta^{-1}$ is smooth on $S(\infty)$ and $\beta^{-1} \in C^{1,\delta}(S(\infty))$ for any positive $\delta.$ In fact the lemma shows a slightly stronger result than that stated in Theorem \ref{thm:Main} as we only used $\beta^{-1} \in C^{1,\delta}(S(\infty))$ in our proof.}
%



\end{document}